\documentclass{article}
\usepackage[utf8]{inputenc}
\usepackage[english]{babel}
\usepackage[a4paper, total={6in, 8in}]{geometry}

\usepackage{indentfirst,enumerate,amsthm}
\usepackage{amssymb,amsmath,mathrsfs,dsfont,graphicx}
\usepackage{appendix}
\usepackage{xcolor}
\usepackage{hyperref}
\hypersetup{
    colorlinks=true,
    linkcolor=blue,
    filecolor=magenta,      
    urlcolor=cyan,
    pdftitle={Overleaf Example},
    pdfpagemode=FullScreen,
    }

\DeclareMathOperator{\im}{Im}
\DeclareMathOperator{\loc}{loc}

\DeclareMathOperator*{\supp}{supp}

\DeclareMathOperator{\Div}{div}
\DeclareMathOperator{\ric}{Ric}
\DeclareMathOperator{\prob}{Prob}
\DeclareMathOperator{\Exp}{Exp}
\DeclareMathOperator{\Log}{Log}
\DeclareMathOperator{\ksd}{KSD}
\DeclareMathOperator{\hess}{Hess}

\theoremstyle{plain}
\newtheorem*{theorem*}{Theorem}
\newtheorem{theorem}{Theorem}[section]
\newtheorem{corollary}{Corollary}[theorem]
\newtheorem{lemma}[theorem]{Lemma}
\newtheorem{proposition}{Proposition}[section]

\theoremstyle{remark}
\newtheorem{definition}{Definition}[section]
\newtheorem*{remark}{Remark}
\newtheorem{example}{Example}[section]
\title{A Framework for Improving the Characterization Scope of Stein's Method on Riemannian Manifolds
}
\author{Xiaoda Qu$^\dagger$ and Baba C. Vemuri$^{\ddag}$\\$^{\dagger}$Department of Statistics \ \ $^{\ddag}$Department of CISE\\University of Florida}

\begin{document}

\maketitle

\begin{abstract}

Stein's method has been widely used to achieve distributional approximations for probability distributions defined in Euclidean spaces. Recently, techniques to extend Stein's method to manifold-valued random variables with distributions defined on the respective manifolds have been reported. However, several of these methods impose strong regularity conditions on the distributions as well as the manifolds and/or consider very special cases. In this paper, we present a novel framework for Stein's method on Riemannian manifolds using the Friedrichs extension technique applied to self-adjoint unbounded operators. This framework is applicable to a variety of  conventional and unconventional situations, 
including but not limited to, intrinsically defined non-smooth distributions, truncated distributions on Riemannian manifolds, distributions on incomplete Riemannian manifolds, etc. Moreover, the stronger the regularity conditions imposed on the manifolds or target distributions, the stronger will be the characterization ability of our novel Stein pair, which facilitates the application of Stein's method to problem domains hitherto uncharted. We present several  (non-numeric) examples illustrating the applicability of the presented theory.

\end{abstract}

\section{Introduction}
Stein's method is a theoretical tool first introduced by Stein in \cite{stein1972bound} to estimate similarity between probability distributions, which has led to numerous applications of distributional approximation to various fields. Several reviews on this subject exist in literature and we cite just a few here \cite{anastasiou2021stein,barbour2014steins,ley2017stein,ross2011fundamentals}.
Classical works mainly focus on discrete or continuous distributions on $\mathbb{R}^n$. Further, the existing methods for distributional comparisons lack a general and universal characterization of the Stein's operator which is applicable to nonlinear spaces, more specifically, Riemannian manifolds. 

In recent times, manifold-valued data have proliferated the world of statistical data science and hence there is need for generalization of classical statistical methods such as the Stein's method to nonlinear spaces specifically, Riemannian manifolds. Statistics on Riemannian manifolds also called Geometric Statistics is a fast growing field with numerous applications. 

Over the past few decades, the need for statistical inference applied to geometric objects (manifold-valued random variables) has motivated several researchers to develop statistical techniques of inference for manifold-valued data. The key challenge in developing these generalizations to suit non-Euclidean geometries is the lack of vector-space structure. Most research efforts over the decades has been focused on theoretical and computational methods for defining and estimation of the first and second order moments of distributions, hypothesis testing, regression analysis etc. from samples of the manifold-valued random variables. There is a large body of literature in this context and we cite a few of the works here \cite{afsari2011,bhattacharya2008statistics,Chakraborty-VemuriAOS19,chikuse2003statistics,Fletcher2007,Kendall10,pennec2006intrinsic,Salem-Vemuri,su2014statistical}. In comparison to this large body of literature, generalization of Stein's method for distributional comparison on Riemannian manifolds has been relatively scarce.

To this end, the key and significant challenges lie in the generalization of Stein's method to manifolds with weak regularity conditions and distributions that may not be smoothly defined everywhere on the manifold. 
In this work, we will elaborate on these challenges and present a novel theory that generalizes the Stein's method to tackle the problem of comparing  unconventional as well as conventional distributions on general Riemannian manifolds that are not necessarily complete or compact. To the best of our knowledge, this is the first such generalization of Stein's method which subsumes relevant results in recent works \cite{barp2018riemann,le2020diffusion,thompson2020,xu2021interpretable} on this topic.

\subsection{Context}

In this section, we begin with the precise statement of the Stein's method and then describe the challenges faced in the generalization from the vector space case to the Riemannian manifold case.

Suppose $M$ is a Riemannian manifold and $\mathcal{P}(M)$ is the space of all Borel probability measures on $M$. If $X$ is a $M$-valued random variable (r.v.), let $Q_X$ denote the pushforward distribution of $X$ on $M$, defined by $Q_X(A):=\prob\{X\in A\}$ for any Borel subset $A$.

\begin{definition}[Stein pair] Given $P\in\mathcal{P}(M)$, a \emph{Stein pair} is a pair $(\mathcal{A}_P,\mathcal{H}_P)$ that consists of a \emph{Stein class} $\mathcal{H}_P$ of real-valued functions on $M$ and a \emph{Stein operator} $\mathcal{A}_P$ that maps the elements in $\mathcal{H}_P$ to real-valued functions on $M$, such that for any $M$-valued r.v. $X$,
$$ Q_X= P\ \Longrightarrow\ E[(\mathcal{A}_P f)(X)]=0,\ \forall f\in \mathcal{H}_P. $$
 We say the Stein pair $(\mathcal{A}_P,\mathcal{H}_P)$ \emph{characterizes} $P$ if the reverse implication $\Longleftarrow$ holds and \emph{characterizes the weak convergence (to $P$)} if $E[(\mathcal{A}_P f)(X_n)]\to 0$ for all $f\in\mathcal{H}_P$ will imply that $Q_{X_n}$ weakly converges to $P$ for a sequence of r.v. $X_n$. 
\end{definition}
\begin{remark} For the rest of the paper, with slight abuse of terminology, we will use the term weak convergence to mean weak convergence to $P$. The usage will be context dependent and will not be confusing.
\end{remark}

The main idea of Stein's method is, given a distribution $P\in \mathcal{P}(M)$ and a Stein pair $(\mathcal{A}_P,\mathcal{H}_P)$ that characterizes $P$ or the weak convergence, it is reasonable to expect $Q_X\approx P$ if $E[(\mathcal{A}_P f)(X)]\approx 0$ for all $f\in\mathcal{H}_P$. Surprisingly, such values $E[(\mathcal{A}_P f)(X)]$ are usually easier to tackle and bound compared to classical probability metrics, and thus can serve as a replacement for classical probability metrics to measure the similarity between $P$ and $Q_X$.

The difficulty is usually not in finding a Stein pair, but rather to enlarge $\mathcal{H}_P$ so that it is rich enough to characterize $P$ or the weak convergence, which is usually established by solving the Stein's equation: For all $P$-integrable function $h$, find an $f_h\in\mathcal{H}_P$ such that $ h-E h(X)= \mathcal{A}_P f_h $. For the distributions on $\mathbb{R}$, such an equation is merely an ordinary differential equation (ODE) for continuous distributions or linear equation for discrete distributions, which can be solved easily. When it comes to the multivariate distributions, the equation becomes a partial differential equation (PDE) and hence difficult to solve in general, which relies on specific properties of the operator, as well as necessary regularity conditions on $M$ and $P$. Many existing works compromise on the characterization and resort to a Stein pair that can only distinguish between $P$ and a $Q$ with $C^1$ Radon-Nikodym derivative $\frac{d Q}{d P}$ w.r.t $P$. On Riemannian manifolds, the situation becomes even more difficult due to the lack of linear structure and a global coordinate chart. 

The generality of distributions that can be distinguished from $P$ using the Stein pair is often described in terms of the conditions imposed on $Q$. In this work, this generality will henceforth be called the \emph{characterization scope} \label{scope}.

\subsection{Prior Works on Stein's Method for Riemannian Manifolds}

Recently, several researchers combined Stein's method with different mathematical tools and developed extensions of the method to more general spaces and presented some relevant applications. For instance,  Barp et al., \cite{barp2018riemann}
introduced a numerical technique for approximating the posterior expectations using a Stein reproducing kernel. Their main contribution was a novel combination of Stein's method with the theory of reproducing kernels Hilbert space and the Sobolev space on Riemannian manifolds.
Later we will show that our framework provides a different path to establish the same result (see example \ref{KSD}).
 Le \& Lewis \cite{le2020diffusion} and Thompson \cite{thompson2020} generalized the classical diffusion approach (first introduced in \cite{barbour1990stein}) to complete Riemannian manifolds assuming that the Bakry-Émery-Ricci curvature is bounded below by some positive constant. They obtained bounds on the Wasserstein metric between distributions with finite first order moment by constructing a Feller diffusion process for which the target measure $P$ is an invariant measure. 
 Hodgkinson et al. in \cite{hodgkinson2020reproducing} constructed a Stein operator on general Polish spaces using an ergodic Markov process. They combine it with the theory of reproducing kernel Hilbert spaces to present a universal theoretical framework for Stein importance sampling, and provide several sufficient conditions for the reproducing Stein kernel to yield a convergence determining kernel Stein discrepancy (KSD). 
In \cite{xu2021interpretable}
Xu and Matsuda developed goodness-of-fit and interpretable model-criticism methods using first and second order Stein's operators and reproducing kernel Hilbert space on general Riemannian manifolds. However, their construction of Stein's operator uses local coordinates and thus depends on the choice of the local chart. Hence, their method can only be applied to distributions supported inside such a local chart. 

It should be noted that one of the key advantages of Stein's method in general is that it does not require the normalization factor of the densities for comparison. This advantage carries over to the case of manifold-valued random variables. This key feature plays an important role from a computational perspective, since most densities on Riemannian manifolds have normalizing factors that are not in closed form and/or are hard to compute. This is one of the major advantages of the Stein's method of distributional comparison over other traditional methods.

\subsection{Our work} 

In this work, we assume $M$ is a connected Riemannian manifold, and $P$ is the target measure supported on entirety of $M$. We assume the density of $P$ w.r.t the volume measure $v$ is known up to a normalizing constant, denoted by $e^{-\phi}$. As mentioned earlier, solving Stein's equation requires the Stein's operator to possess certain good properties, so we settle on  
$$\mathcal{L}_P f=\Delta f-g(\nabla \phi,\nabla f),$$
also known as the weighted Laplacian under the weight $P$, which possesses the properties we seek namely, symmetry and negativity (negative definiteness) of the operator. The operator $\mathcal{L}_P$ was also adopted in \cite{barp2018riemann,le2020diffusion,thompson2020}, but in our work here, we take a completely different approach based on extending the domain of definition of $\mathcal{L}_P$ to a larger set, using Friedrichs extension \cite[Theorem X.23]{reed1975ii} of the self-adjointness property. This facilitates in achieving a broader characterization scope of $\mathcal{L}_P$ and the Stein class $\mathcal{H}_P$ so that $(\mathcal{L}_P,\mathcal{H}_P)$ can characterize $P$ or the weak convergence, defined earlier.
The salient features of our approach can be summarized  as follows:
\begin{itemize}
\item  {\bf New insights into Stein's method:} The key feature of Stein's method proposed here involves providing a framework for analyzing the characterization scope of the Stein pair that facilitates distributions on Riemannian manifolds to be compared. The stronger the imposed regularity conditions are, which are described in terms of the the type of manifolds and the conditions imposed on the target distributions, the stronger the characterization ability of the Stein pair. To the best of our knowledge, this framework subsumes all existing results on Stein pair characterization on Riemannian manifolds.

\item {\bf Mild regularity conditions on manifolds:}  Unlike earlier works discussed above \cite{barp2018riemann,le2020diffusion,thompson2020}, our method is applicable to incomplete manifolds, so it permits $P$ to be supported on some open connected subset of $M$. Further, our method is applicable to manifolds with boundary. This in turn allows us to specify distributions $P$ that are truncated and thus are supported on some bounded region of the manifold (see example \ref{Truncation}). For instance, truncated distributions are often encountered in the computation of statistics, specifically when Fréchet mean (FM) of samples is required to be unique (see  \cite{afsari2011,Kendall10,pennec2006intrinsic,groisser2004newton} for uniqueness conditions). All these properties collectively broaden the scope and applicability of our generalization of the Stein's method. 

\item {\bf Mild regularity on the target measure $P$:} In our work, the regularity assumptions on target measure $P$ imposed in the recent works presented in \cite{barp2018riemann,le2020diffusion,thompson2020} is further weakened, making our approach much more general and applicable. For instance, distributions with locally Lipschitz density fall into this category. A typical example of such distributions includes intrinsically defined distributions (see example \ref{Intrinsic}). Intrinsic distributions inherit several properties of  Euclidean distributions, e.g., entropy maximization \cite{pennec2006intrinsic}, law of large numbers \cite{bhattacharya2003large}, consistency between MLE and FM \cite{cheng2013novel,pennec2006intrinsic} and so on. The explicit form of such distributions often rely on geodesic distance functions $d(\cdot,e)$, the inverse exponential map $\log_e(\cdot)$, the volume density function $\theta(p)$ \cite[p.154]{besse1978manifolds}, and thus are not differentiable or globally continuous.
\end{itemize}

The rest of the paper is organized as follows. In \S\ref{Background} we present the background mathematics needed to follow the rest the of paper, specifically, concepts from  differential geometry, measure theory, weighted Sobolev spaces and the theory of unbounded operators. In \S \ref{MainResults}, we present the main theoretical results of the paper, namely the general framework and the key theorems for Stein's method of distributional approximation on Riemannian manifolds. Several examples are then presented in \S\ref{applications} illustrating the use of the framework in characterization of distributional approximation on compact, non-compact and incomplete manifolds.  We draw conclusions in \S\ref{Conc}. Proofs of all the theorems are included in the \hyperref[Appendix]{Appendix}.

\section{Desiderata: Riemannian Geometry, Measure Theory, Functional Analysis and Unbounded Operators}\label{Background}

In this section, we set up the notations and standing assumptions, and then present a brief introduction to mathematical definitions and concepts that will be used in the rest of the paper. The definitions, assumptions and theorems/propositions pertain to the fields of Geometric Analysis and Unbounded Operator Theory involving topics from Differential Geometry, Measure Theory and Functional Analysis. For a more comprehensive study of the background material presented here, we refer the readers to \cite{lee2006riemannian,lee2013smooth,petersen2016riemannian} on differential geometry, \cite{billingsley2008probability,billingsley2013convergence} for measure theory on metric spaces, \cite{adams2003sobolev,aubin2012nonlinear,hebey2000nonlinear,taylor2000PDEI} for weighted Sobolev spaces on manifolds and \cite{reed1972methods,reed1975ii} for unbounded operator theory.

\subsection{Riemannian Geometry}
For a detailed account on various definitions of geometric quantities provided below, we refer the reader to a standard text book on Differential Geometry for instance, \cite{lee2006riemannian,lee2013smooth,petersen2016riemannian}.

Let $\mathbb{R}_+^n:=\{(x^1,\cdots,x^n)\in \mathbb{R}^n: x^n\geq 0\}$. 
A topological space $X$ is \emph{locally $\mathbb{R}_+^n$}, if for each $x\in X$, there exists a homeomorphism $\xi$ between an open neighborhood $U$ of $x$ and an open set $\xi(U)$ in $\mathbb{R}_+^n$. Such a pair $(U,\xi)$ is a \emph{local chart} or \emph{chart}. 

\begin{definition}[Manifold] A $n$-manifold $M$ is a second countable Hausdorff locally $\mathbb{R}_+^n$ space. The boundary of $M$, denoted $\partial M$, is the subset locally mapped to $\partial\mathbb{R}_+^n:=\{(x^1,\cdots,x^n)\in\mathbb{R}_+^n: x^n= 0\}$ by the chart. A manifold $M$ is said to be \emph{with boundary} if $\partial M\neq \emptyset$, \emph{without boundary} if $\partial M=\emptyset$.
\end{definition}

The smoothness of manifolds can be explained via the concept of an atlas (a collection of charts). A \emph{smooth atlas} is a family of charts $\{(U_\alpha,\xi_\alpha)\}_\alpha$ such that $M=\bigcup U_\alpha$ and each $\xi_\alpha\circ \xi_\beta^{-1}$ is smooth in $\mathbb{R}_+^n$. 

\begin{definition}[Smooth Manifolds] A manifold is \emph{smooth} if it admits a smooth atlas.
\end{definition}

Smoothness facilitates the definition of  differentiable maps, tangent vectors and tensors on manifolds. A map from a smooth manifold to another smooth manifold is \emph{differentiable}, \emph{continuous differentiable} or $C^k$ if its composition with charts are respectively differentiable, continuous differentiable and $C^k$. Note that this definition is independent of the choice of charts and coincides with the classical Euclidean notions, since the compositions of chart maps are smooth. Let $C^k(M),C^k_b(M), C^k_c(M)$ be the spaces of $C^k$, bounded $C^k$ and compactly supported $C^k$ functions on $M$. The notation $C^L(M)$ and $C^L_{\loc}(M)$ represents the space of Lipschitz continuous and locally Lipschitz continuous functions, with respect to the Riemannian distance $d$ about to be defined next. When $\partial M\neq \emptyset$, $C^k_c(M)$ refers to compactly supported $C^k$-functions with $f|_{\partial M}=0$.

\subsubsection*{Tangent and cotangent tensors} A \emph{local curve} at a point $x$ of the manifold $M$ is a smooth map $\mathfrak{c}:(-\epsilon,\epsilon)\to M $ with $\mathfrak{c}(0)=x$. A \emph{tangent vector} at $x$ is an equivalence class of local curves at $x$ with the same gradient $\frac{d}{d t}(\xi\circ \mathfrak{c})(0)$ for some chart $(U,\xi)$. The \emph{tangent space} $T_x M$ is the $n$-dimensional vector space of tangent vectors at $x \in M$. The \emph{cotangent space} $T_x^* M$ is then the dual space of $T_x M$. For a tangent vector $Y\in T_x M$ and a smooth functions $f$, we have $Y (f)=\frac{d }{d t} f\circ\mathfrak{c}(0)$ where $\mathfrak{c}$ is a local curve corresponding to $Y$.

\subsubsection*{Tensors} The \emph{tensor product space} of $(k,l)$-type tensors on $T_x M$ is denoted by $T^{k,l}_x M$, whose  elements are \emph{tangent tensors} at $x$. A \emph{vector field} or a \emph{tensor field} is an assignment that assigns to each $x\in M$ a vector in $T_x M$ or a tensor in $T^{k,l}_x M$. The vector or tensor field is \emph{smooth} or \emph{measurable} if the assignment is smooth or measurable.

\subsubsection*{Differential of maps} Given two smooth manifolds $M$ and $N$, the \emph{differential} of a smooth map $f:M\to N$ at each $x\in M$ is the linear map $d f_x: T_x M\to T_{f(x)} N$ that maps the equivalence class of local curve $\mathfrak{c}$ at $x$ to the equivalence class of local curve $f\circ \mathfrak{c}$ at $f(x)$. Specifically, when $N$ is the real line $\mathbb{R}$, $d f$ is a linear function on $T_x M$. When $M$ is some real interval $(a,b)$, the map is then a smooth curve, usually denoted by  $\gamma:(a,b)\to M$, and the tangent vector $\gamma'(t)$ at $t$ is the equivalence class of $\gamma$ as a local curve itself.

\subsubsection*{Riemannian metric} A \emph{Riemannian metric} $g$ is a smooth $(0,2)$-tensor field with $g_x\in T_x^{0,2}M$ symmetric and positive definite for each $x\in M$, that is, $g$ makes every tangent space, $T_x M$, an inner product space. The inner product on a vector space $V$ naturally induces an inner product structure on the tensor product space of $(k,l)$-type tensors on $V$ \cite[\S II.4]{reed1972methods}.  Therefore, $g$ induces an inner product on each $T^{k,l}_x M$, also denoted here by $g$. The Riemannian metric gives meaning to the pointwise length of vector or tensor fields.

\begin{definition}[Riemannian manifold] A \emph{Riemannian manifold} $M$ is a smooth manifold endowed with a Riemannian metric $g$.
\end{definition}

\subsubsection*{Riemannian distance} For a connected Riemannian manifold $M$, any two points $x,y$ can be connected by a piecewise smooth curve $\gamma:[0,1]\to M$ with $\gamma(0)=x,\gamma(1)=y$. The length of the curve is then defined as $L(\gamma)=\int_0^1 |\gamma'(t)| d t $. The Riemannian distance $d$ is then given by $$d(x,y)=\inf \left\{L(\gamma):\text{for all } \gamma \text{ connecting } x,y\right\}.$$ A Riemannian manifold $M$ is \emph{complete} when $M$ is a complete metric space endowed with the distance $d$. The Hopf-Rinow theorem \cite[Theorem 5.7.1]{petersen2016riemannian} states that a Riemannian manifold is complete if and only if each bounded and closed subset of $M$ is compact.

\subsubsection*{Connection and Covariant derivative} 

A connection is a map $(Y,Z)\mapsto \nabla_Z Y$ that is the derivative of a vector field $Y$ in the direction of another vector field $Z$ on a Riemannian manifold $M$. The fundamental theorem of Riemannian geometry states that there is an unique connection on any Riemannian manifold, called the \emph{Levi-Civita connection}, that is torsion-free and compatible with the Riemannian metric. Imposition of the metric compatibility constraint along with tensor product and tensor contraction rules allows one to extend such derivatives $\nabla_Z Y$ to higher order cases where $Y$ is a $(k,l)$-type tensor field. In such cases, $\nabla_Z Y$ is called the \emph{covariant derivative} of $Y$ in the direction of $Z$.  Moreover, $\nabla Y$ can be regarded as a $(k,l+1)$-type tensor field since one can substitute in any vector $Z$, which is called the \emph{total covariant derivative} of $Y$. For further details, we refer the readers to \cite[\S 5]{lee2006riemannian}. However, in this work, the readers are only required to know that the covariant derivative is the generalization of higher order Euclidean derivatives to Riemannian manifolds.

\begin{remark} When the symbol $\nabla$ stands for covariant derivative, $\nabla f = d f$ for a differentiable function $f$. However, $\nabla$ stands for the gradient operator in some other works and $\nabla f$ is the vector such that $d f(\cdot) =g(\cdot,\nabla f)$. They differ up to a musical isomorphism \cite[p.342]{lee2013smooth}. In this work, since we've already introduced that the Riemannian metric $g$ will induce the same structure on the cotangent space $T^*_x M$, such abuse in notation will not lead to any misinterpretation. Thus, we will not distinguish between them here.
\end{remark}

\subsubsection*{Volume measure} The volume measure $v$ on $M$ can be defined as the $n$-dimensional Hausdorff measure or defined through integration of differential forms. These two definitions coincide.  For details we refer the readers to \cite{evans2018measure,lee2013smooth}. In this work, the reader are only required to know that the volume measure is the generalization of Euclidean Lebesgue measure to Riemannian manifolds, and commonly plays the role of a dominating measure of probability densities functions on the manifolds. The boundary $\partial M$, as a smooth submanifold of $M$, naturally inherits a $n-1$-dimensional volume measure from $v$, denoted by $\partial v$.

\subsubsection*{Differential operator} The well known divergence and Laplacian operators in Euclidean space have a generalization to Riemannian manifolds, also denoted here by the same symbols $\Div$ and $\Delta$ respectively. The Laplacian is called the Laplace-Beltrami operator in the manifold case. Specifically, for functions $f$ and $h$ on $M$, following property will be used frequently in \S\ref{SPE}:
\begin{equation}\label{DIVfml}
    \Div(h\nabla f)= h\cdot \Delta f +g(\nabla f,\nabla h). 
\end{equation}
Moreover, the following theorem \cite[Theorem 16.32]{lee2013smooth} will be used repeatedly in \S\ref{SPE}.
\begin{theorem*}[Divergence theorem] \label{DivThm}  For a compactly supported $C^1$ vector field $Y$,
$$\int_M \Div(Y) d v =\int_{\partial M} g(\Vec{n},Y) d\partial v,$$
where $\Vec{n}$ is the outward-pointing unit vector field along $\partial M$.
\end{theorem*}

\subsubsection*{Geodesic and the Exponential map} A \emph{geodesic} is a curve $\gamma:(a,b)\to M$ such that $\nabla_{\gamma'(t)}\gamma'(t)=0$ for all $t\in (a,b)$. The geodesic is the generalization of Euclidean straight line to Riemannian manifolds. For each $V\in T_x M$ the tangent space at $x\in M$ that is small enough, there exists an unique geodesic $\gamma(t)$ such that $\gamma(0)=x$ and $\gamma'(0)=V$. The \emph{exponential map} $\Exp_x: T_x M\to M$ is then defined as the map $V\mapsto \gamma(1)$. The inverse of $\Exp_x$, called \emph{logarithm map}, is denoted $\Log_x$. For more details, we refer the reader to \cite{lee2013smooth}.

\subsubsection*{Cut locus} \cite[\S VIII.7]{kobayashi1963foundations} The \emph{cut locus} $\mathscr{C}_\mu$ of a point $\mu$ in the Riemannian manifold $M$ is the set where $d(\cdot,\mu)^2$ is not smooth. The exponential map $\Exp_\mu$ is then a diffeomorphism between an open cell of $T_\mu M$ and $\mathscr{N}_\mu:= M\setminus \mathscr{C}_\mu$, so is the logarithm map $\Log_\mu$. It is well-known that the volume measure of $\mathscr{C}_\mu$ is $0$.

\subsection{Measure Theory on Metric spaces}

Let $M$ be a Riemannian manifold. For a subset $E$ of $M$, let $\overline{E}$ denote its closure and $E^\circ$ denote its interior. We write $E_1\Subset E_2$ if $\overline{E}_1\subset E^\circ_2$ and $\overline{E}_1$ is compact. Let $\mathcal{M}(M)$ and $\mathcal{P}(M)\subset\mathcal{M}(M)$ denote the spaces of finite measures and probability measures respectively on the Borel algebra $\mathscr{B}(M)$. In this work, given a measure $Q$, we say a set $F\in\mathscr{B}(M)$ is $Q$-{\em null} if $Q(F)=0$. When an event happens excepting on a $Q$-null set, it is said to be $Q$-almost everywhere, denoted $Q$-a.e.

To reduce clutter and to simplify the notation, let $ Q(f;F)$ represent the integration $\int_F f d Q $ for $Q\in\mathcal{M}(M)$, $F\in\mathscr{B}(M)$ and measurable function $f$, abbreviated as $Q(f)$ if $F=M$.

\subsubsection*{Absolute continuity} For two measures $Q_1,Q_2\in\mathcal{M}(M)$, $Q_1$ is said to be \emph{absolutely continuous} with respect to $Q_2$, and denoted by $Q_1\ll Q_2$, if any $Q_2$-null set is $Q_1$-null. The Radon-Nikodym theorem \cite[theorem 32.2]{billingsley2008probability} states that $Q_1\ll Q_2$ if and only if there exists a non-negative measurable function, denoted by $\frac{d Q_1}{d Q_2}$, such that $Q_1(F)=Q_2\big(\frac{d Q_1}{d Q_2};F\big)$ for all $F\in\mathscr{B}(M)$. We write $Q_1\sim Q_2$ if $Q_1\ll Q_2$ and $Q_2\ll Q_1$. Further, we write $Q_1\asymp Q_2$ if there exists $C>0$ such that $C^{-1} Q_1(F)\leq Q_2(F)\leq C Q_1(F)$ for all $F\in\mathscr{B}(M)$.

\subsubsection*{Weak convergence} For $Q_n,Q\in \mathcal{P}(M)$, say $Q_n$ converges weakly to $Q$, denoted $Q_n\xrightarrow{w.} Q$, if $Q_n(f)\to Q(f)$ for all $f\in C_b(M)$.

\subsubsection*{$L^p$-spaces} The $L^p$ space is denoted by $L^p(M;Q)$ for $p\geq 1$ and measure $Q\in \mathcal{M}(M)$. Moreover, $L^p_{\loc}(M;Q)$ denote the space of locally $L^p$-integrable functions under $Q$, and $L^p_0(M;Q)$ is the center $L^p$ space, that is, the space of all $f\in L^p(M;Q)$ with $Q(f)=0$. The term $L^p(M;Q)$ is usually abbreviated as $L^p(Q)$, since we only discuss measures on Riemannian manifold $M$ in this work. Specifically, $L^p(M)$ represents the $L^p$ space under the volume measure $v$.

\subsubsection*{$L^p$-tensor field} The $L^p$ norm of tensor fields can be similarly defined to that of $L^p$ norms of scalar functions. 
Suppose $X$ is a $(k,l)$-type tensor field (note that a $(1,0)$-type tensor field is a vector field). Let $|X|:=\sqrt{g(X,X)}$ represent the pointwise length of $X$ and $\Vert X \Vert_{L^p(Q)}:=\sqrt[{}^p]{Q(|X|^p)}$ be the integral $L^p$-norm under the measure $Q$. In this work, we don't use special notation for the space of $L^p$ tensor fields.
\vspace{5mm}

It is worth noting that the $L^2$ space is a Hilbert space endowed with inner product structure. In order to incorporate this inner product structure into our notation for measurable functions and tensor fields, we introduce the following notations. Let $Q(X,Y):=Q(g(X,Y))=\int_M g(X,Y) d Q$ for two tensor fields $X$ and $Y$ of the same type (matching indices). Specifically, let $Q(f,h):=Q(f\cdot h)$ for functions $f$ and $h$.

\subsection{Weighted Sobolev Spaces}

In this section, we introduce the concept of weighted Sobolev spaces, which will be needed in \S\ref{SPE}, where we present the Stein pair extension to accommodate larger function classes and an appropriate Stein operator for such classes.
Let
$$\mathcal{C}^{p,k}(M;Q):=\left\{f\in C^\infty(M): Q(|\nabla^j f |^p)<+\infty, 0\leq j\leq k\right\},$$ 
then the weighted Sobolev space $W^{p,k}(M;Q)$ is defined as the completion of $\mathcal{C}^{p,k}(M;Q)$ under the Sobolev norm 
$$\Vert f\Vert_{W^{p,k}(M;Q)}:= \bigg[\sum_{j=0}^k Q(|\nabla^j f|^p)\bigg]^{\frac{1}{p}}.$$

The local Sobolev space $W^{p,k}_{\loc}(M;Q)$ is then
defined as the space of functions $f$ such that for each $x\in M$, there exists an open neighborhood $x\in U\Subset M$ with a smooth boundary such that $f|_{\overline{U}}\in W^{p,k}(\overline{U};Q)$, or equivalently, $f|_{\overline{\Omega}}\in W^{p,k}(\overline{\Omega};Q)$ for any open $\Omega\Subset M$ with a smooth boundary. Let $W^{p,k}_c(M;Q)$ represents the subspace of functions in $W^{p,k}(M;Q)$ with compact support.

When $p=2$, the space $W^{2,k}(M;Q)$ is denoted by $H^k(M;Q)$, which is a Hilbert space equipped with the inner product
$$ \langle f,h\rangle_{H^k(M;Q)}=\sum_{j=0}^k Q(\nabla^j f,\nabla^j h). $$

The abbreviation used here is similar to that in the $L^p$ spaces. We abbreviate $W^{p,k}(M;Q)$ by $W^{k,l}(Q)$ and $H^k(M;Q)$ by $H^k(Q)$ when there is no confusion with regards to the domain, and $W^{p,k}(M)$ , $H^k(M)$ represent the Sobolev spaces with the volume measure $v$.

Following propositions are frequently used in this work.

\begin{proposition}\label{LisSob} Suppose $M$ is a Riemannian manifold and $p\geq 1$.
\begin{enumerate}
    \item $C^L(M)\subset W^{p,1}(M)$ for compact $M$ with or without boundary.
    \item $C^L_{\loc}(M)\subset W^{p,1}_{\loc}(M)$ for arbitrary $M$.
\end{enumerate}
\end{proposition}

\begin{proof} The case of compact manifolds without boundary directly follows from \cite[proposition 2.4]{hebey2000nonlinear}. Moreover, a compact manifold with boundary can be embedded in a compact manifold without boundary of the same dimension (see the argument in  \cite[\S 4.4]{taylor2000PDEI}). A variation of Tietze extension theorem \cite{mcshane1934extension} states that a Lipschitz continuous function on closed subset can be extended to a Lipschitz continuous functions globally, so the result also works for compact manifold with boundary. For an arbitrary manifold $M$, we can just take a compact local neighborhood and draw the conclusion. 
\end{proof}

\begin{proposition}[Poincaré Inequality]\label{PI} Suppose $M$ is compact with boundary (possibly empty). There exists $C>0$, depending on $p\in [1,n)$ and $M$, such that
$$ \Vert f - (f)_M \Vert_{L^p(M)}\leq C\Vert f \Vert_{W^{p,1}(M)} $$ 
for $f\in W^{p,1}(M)$, where $(f)_M=[v(M)]^{-1}\int_M f d v$ is the average of $f$ over $M$. 
\end{proposition}

\begin{proof} See \cite[Theorem 2.10]{hebey2000nonlinear}
\end{proof}

\subsection{Unbounded Operators}
In this section, we present a brief note on unbounded linear operators on Hilbert spaces, which will be used in \S\ref{SPE} for extension of existing Stein operators to more general class of functions. 

Let $\mathcal{H}$ be a Hilbert spaces. An unbounded operator $A: \mathcal{H} \rightarrow \mathcal{H}$
is a linear map from a densely defined proper subset $D(A) \subsetneq \mathcal{H}$ -- the domain of $A$ -- to $\mathcal{H}$. Note that, in contrast, the bounded operator can be defined everywhere on the domain $D(A)=\mathcal{H}$.

\subsubsection*{Extension} An {\em extension} of $A$ is an operator $\Tilde{A}$ such that $D(A)\subset D(\Tilde{A})$ and $\Tilde{A}|_{D(A)}=A$. We write $A\subset \Tilde{A}$ if $\Tilde{A}$ is an extension of $A$. The {\em graph} $G(A)$ of $A$ is the set $G(A):=\{(x,A x):x\in D(A)\}\in\mathcal{H}\times\mathcal{H}$.  If the closure $\overline{G(A)}$ of $G(A)$ in $\mathcal{H}\times\mathcal{H}$ is a graph of an operator, this operator is called the {\em closure} of $A$, denoted by $\overline{A}$. That is, $\overline{A}$ is defined by $G(\overline{A})=\overline{G(A)}$.

\subsubsection*{Adjoint} Given an unbounded operator $A$, define the set
$$ D(A^*):=\left\{\xi\in \mathcal{H}: \text{ The map } x\mapsto\langle A x,\xi\rangle \text{ is bounded on } D(A) \right\}. $$
For $\xi\in D(A^*)$, since the map $x\mapsto\langle A x,\xi  \rangle$ is bounded on a dense subset of $\mathcal{H}$, it extends to a bounded linear map on $\mathcal{H}$. By Riesz representation theorem, such a linear map can be identified with a unique $\eta\in\mathcal{H}$ such that $\langle  A x,\xi \rangle=\langle x,\eta \rangle$ for all $x\in D(A)$. The map $\xi\mapsto \eta$ is then defined as the adjoint of $A$, denoted $A^*$, whose domain is $D(A^*)$.

\vspace{5mm}

An unbounded operator $A$ is {\em closed} if $G(A)$ is closed in $\mathcal{H}\times\mathcal{H}$; {\em closable} if $A$ admits a closed extension; {\em symmetric} if $\langle x,A y\rangle=\langle Ax, y\rangle$ for all $x,y\in D(A)$; {\em self-adjoint} if $A$ is symmetric and $A=A^*$; {\em positive (negative)} if $\langle x,Ax\rangle\geq(\leq)\ 0$ for all $x\in D(A)$.

\begin{proposition}\label{UnbOpe} Suppose $A$ is an unbounded operator on $\mathcal{H}$. Then:
\begin{enumerate}
    \item $A^*$ is always closed. 
    \item $A$ is closable $\iff$ $\overline{T}$ exists $\iff$ $D(A^*)$ is dense, in which case $\overline{T}=T^{**}$. 
    \item If $A$ is symmetric, then $A$ is closable and $A\subset\overline{A}\subset A^*$.
    \item If $A$ is closed, $\ker A={\im A^*}^\perp$.
\end{enumerate}
\end{proposition}

\begin{proof} See \cite[\S VIII.1 \&\ \S VIII.2]{reed1972methods}.
\end{proof}

\section{Main Results}\label{MainResults}

Armed with the background material presented in last section, we are now ready to present the main results of the paper.
We temporarily leave the specific form of the Stein's operator aside to discuss a preparatory question before we proceed to our main results. Suppose $(\mathcal{A}_P,\mathcal{H}_P)$ is an arbitrary Stein pair, and we denote the image of $\mathcal{A}_P$ by $\mathcal{I}_P:=\mathcal{A}_P(\mathcal{H}_P)\subset L_0(P)$. In addition, to utilize the inner product structure, we assume $\mathcal{I}_P$ is a subspace of $L^2_0(P)$.

Note that the characterization ability of $(\mathcal{A}_P,\mathcal{H}_P)$ totally depends on $\mathcal{I}_P$. Another way to portray this characterization ability is to seek a solution to the following questions:

\medskip
\noindent\fcolorbox{black}{lightgray}{%
\parbox{\textwidth}{
\begin{itemize}
\item  Does $Q(f)=P(f)$ for all $f\in \mathcal{I}_P$ imply  $Q=P$ ?

\item Does $Q_n(f)\to P(f)$ for all $f\in \mathcal{I}_P$ imply $Q_n\xrightarrow{w.}P$ ?
    \end{itemize} }
}
\medskip

Some of the relevant notions that have been established and standardized in literature \cite[\S 3.4]{ethier2009markov}. Suppose $Q$, $Q'$ and $Q_n$ are measures. A function class $\mathcal{C}$ is,
\begin{itemize}
    \item {\em separating} if $Q'(f)=Q(f)$ for all $f\in\mathcal{C}$ implies $Q'=Q$,
    \item {\em convergence determining} if $Q_n(f)\to Q(f)$ for all $f\in \mathcal{C}$ implies $Q_n\xrightarrow{w.} Q$.
\end{itemize}

Several studies already exist concerning separating and convergence determining classes \cite{blount2010convergence,lohr2016boundedly}, but the conditions they impose are mostly topological and hard to test in our context. 

However, a direct observation will immediately provide a necessary condition. Note that $L^2_0(P)$ is a Hilbert space, thus if $\mathcal{I}_P$ is not dense in $L^2_0(P)$, there would exist a $\rho\in L^2_0(P)$ such that $\rho\perp \mathcal{I}_P$. If such a $\rho$ is bounded below by some $c\in\mathbb{R}$, then $d Q= (1-c^{-1}\rho) d P$ is a probability measure such that $Q(f)=P(f,1-c^{-1}\rho)=0$ for $f\in\mathcal{I}_P$. Therefore, $\mathcal{I}_P$ does not distinguish $P$ and $Q$ and thus is not separating. A similar argument applies when $\rho$ is bounded above. As a consequence, we may as well make $\mathcal{I}_P$ dense in $L^2_0(P)$ if we expect $(\mathcal{A}_P,\mathcal{H}_P)$ to characterize $P$ or weak convergence.

\subsection{Stein Pair Extension}
\label{SPE}

In this section, we present a solution to the above described problems by seeking to extend the Stein pair using the Friedrichs extension, a well established tool from Mathematical Physics \cite{friedrichs1953differentiability}.

We assume throughout that $(M,g)$ is a connected smooth Riemannian manifold with boundary (possibly empty), but not necessarily complete unless explicitly stated. Specifically, let $n=\dim M$. The target distribution $P$ is assumed to have a density $\frac{d P }{d v}\propto e^{-\phi}\in L(M) $ w.r.t the volume measure $v$, known up to a constant. Note that $P\sim v$ on $M$ since $e^{-\phi}>0$. We assume $e^{-\phi}\in H^1_{\loc}(M)$ so that $\nabla \phi:=e^{\phi}\nabla(e^{-\phi})$ makes sense. The operator $\mathcal{L}_P$ is then given by
$$ \mathcal{L}_P =e^\phi\Div(e^{-\phi}\nabla )=\Delta -g(\nabla \phi,\nabla ). $$

As discussed previously, we wish to collect as many functions with $P(\mathcal{L}_P f)=0$ as possible so as to satisfy the condition that the image of $\mathcal{L}_P$ is at least dense in $L^2_0(P)$. In addition to constant functions, we have

\begin{proposition} For $f\in C^2_c(M)$, $P(\mathcal{L}_P f)=0$.
\end{proposition}

\begin{proof} Consider $\Omega\Subset M$ with a smooth boundary such that $\supp{(f)}\Subset \Omega$. Since $e^{-\phi}\in H^1_{\loc}(M)$, there exists $\psi_n\in C^\infty(\overline{\Omega})$ such that $\psi_n\to e^{-\phi}$ in $H^1(\overline{\Omega})$. Applying \nameref{DivThm}, we have $\int_{\overline{\Omega}}\Div(\psi_n\nabla f) d v=0$ since $\nabla f$ is $0$ on the boundary of $\Omega$. Furthermore, applying property \ref{DIVfml}, we get,
\begin{eqnarray*}
\int_{\overline{\Omega}}\Div(\psi_n\nabla f) d v &=& \int_{\overline{\Omega}} \Delta f\cdot \psi_n - g(\nabla \psi_n,\nabla f) d v\\
&\to& \int_{\overline{\Omega}} \Delta f\cdot e^{-\phi} - g(\nabla e^{-\phi},\nabla f) d v\\
&=& \int_M \mathcal{L}_P f\cdot e^{-\phi} d v=P(\mathcal{L}_P f),
\end{eqnarray*}
as $n\to \infty$. Therefore, we conclude that $P(\mathcal{L}_P f)=\lim_{n\to \infty}\int_{\overline{\Omega}}\Div(\psi_n\nabla f) d v=0 $.
\end{proof}

 Other than constant functions and $C^2_c(M)$, it's difficult in general to determine which kind of functions satisfy $P(\mathcal{L}_P f)=0$. Unfortunately, if we just settle on $C^2_c(M)\oplus \mathbb{R}$, the image of $\mathcal{L}_P$ on $C^2_c(M)\oplus \mathbb{R}$ is not dense in $L^2_0(P)$.

\begin{example} Consider the uniform distribution $U$ on the unit ball $B\subset\mathbb{R}^2$, so the corresponding operator is $\mathcal{L}_U=\Delta$. Note that the function $e^x\cos y$ satisfies $\Delta(e^x\cdot\cos y)=0$, i.e., is harmonic. Integrate by parts and we'll have
$$ \int_B \Delta f \cdot e^x\cos y dxdy=\int_B f\cdot \Delta (e^x\cos y) dxdy=0,\quad \forall f\in C^2_c(B)\oplus \mathbb{R}. $$
Therefore, $e^x\cos y\perp \im\mathcal{L}_P$. Since $e^x\cos y>0$ on $B$, we may normalize it to a density function, then $(\mathcal{L}_U,C^2_c(M)\oplus \mathbb{R})$ can not distinguish $U$ from this density!  
\end{example}

To overcome this problem, the definition of $\mathcal{L}_P$ must be extended to accommodate a larger function class. Note that $C^2_c(M)\oplus \mathbb{R}$ is dense in the Hilbert space $L^2(P)$, so we may try to extend the domain of $\mathcal{L}_P$ to a larger subset of $L^2(P)$, so that the image will at least be dense in $L^2_0(P)$. However, since $\mathcal{L}_P$ is unbounded, this extension is not straightforward and in order to accomplish it, we must invoke the theory of unbounded operators on Hilbert space.

Observe that $\Div(e^{-\phi}h\nabla f )=h\mathcal{L}_P f  e^{-\phi}+g(\nabla f,\nabla h) e^{-\phi}$ for $f,h\in C^2_c(M)\oplus \mathbb{R}$ by property \ref{DIVfml}. The \nameref{DivThm} implies the integration of $\Div(e^{-\phi}h\nabla f )$ over volume measure is $0$, which further leads to $P(\mathcal{L}_P f,h )=-P(\nabla f,\nabla h)$. Since $P(\cdot,\cdot)$ is an inner product, we have that $\mathcal{L}_P$ is symmetric and negative(definite) on $C^2_c(M)\oplus \mathbb{R}$. 

In the theory of unbounded operators, symmetry ensures the existence of such aforementioned extensions to the larger subset of $L^2(P)$. Such extensions, nevertheless, are not unique in general, which depends on the specific method of extension used and the choice of the initial domain. The most common method of extensions are the closure $\overline{\mathcal{L}}_P$ or the adjoint $\mathcal{L}_P^*$. However, on the strength of negative definiteness of the operator, we finally adopt the Friedrichs extension stated in the following theorem.

\begin{theorem*}[Friedrichs extension theorem]\label{FriExt}   Let $A$ be a negative symmetric unbounded operator, then there exists an unique self-adjoint extension $\widehat{A}$ of $A$ corresponding to the quadratic form $-q(\phi,\psi)=(\phi,A\psi)$ for $\phi,\psi\in D(A)$.
\end{theorem*}
\begin{proof} See \cite[Theorem X.23]{reed1975ii}.
\end{proof}

Such extension $\widehat{A}$ is called the \emph{Friedrichs extension} of $A$. In our context, 
the operator $A$ in the above theorem refers to the operator $\mathcal{L}_P$ on Hilbert space $L^2(P)$, and corresponding Friedrichs extension will be denoted $\widehat{\mathcal{L}}_P$.

In fact, we have $\overline{\mathcal{L}}_P\subset\widehat{\mathcal{L}}_P\subset\mathcal{L}_P^*$ in general, where $\overline{\mathcal{L}}_P$ is considered to be the smallest extension, but too small to characterize $P$ in our case, and $\mathcal{L}_P^*$ is the largest, but too large so as to let functions failing to satisfy $P(\mathcal{L}_P f)=0$ creep in, while $\widehat{\mathcal{L}}_P$, with a moderate size, serves our needs well.

In the classical theory of unbounded operators on Hilbert spaces, the operator $A$ always comes with a preassigned domain $D(A)$. However, in our approach, one has freedom to choose the domain of $\mathcal{L}_P$. This is because, the initial domain is not only $C^2_c(M)\oplus\mathbb{R}$ but can also be any dense subspace of $L^2(P)$ that satisfies specific conditions. For example, $C^\infty_c(M)$ can serve as an initial domain, on which $\mathcal{L}_P$ satisfies the requirements of Friedrichs extension so that a valid extension can be performed. For both theoretical and practical flexibility, we propose the following general framework:

\medskip
\noindent\fcolorbox{black}{lightgray}{
\parbox{\textwidth}{
\begin{definition}[Extendable] Suppose $\mathcal{H}_P$ is a function class such that
\begin{enumerate}
    \item $\mathcal{H}_P\subset H^2(P)$ is a linear subspace dense in $L^2(P)$,
    \item $ P(\mathcal{L}_P f, h)=-P(\nabla f,\nabla h)$ for all $f,h\in\mathcal{H}_P$,
\end{enumerate}
then, we say $\mathcal{L}_P$ is {\em Friedrichs extendable} (or {\em extendable}) on $\mathcal{H}_P$.
\end{definition} }}
\medskip

The first condition ensures that $\mathcal{L}_P$ is well-defined for functions $\mathcal{H}_P\subset H^2(P)$ and the images $\mathcal{L}_P f$ are $L^2$-integrable. The second condition implies both the symmetrization and negativity of $\mathcal{L}_P$ on $\mathcal{H}_P$. Note that we do not assume that $(\mathcal{L}_P,\mathcal{H}_P)$ is a Stein pair up to this juncture, as it is not necessary for the existence of Friedrichs extension. 

Next, we elaborate the procedure of Friedrichs extension in our context. Suppose $\mathcal{L}_P$ is extendable on $\mathcal{H}_P$. Let $\mathscr{H}_P$ denote the closure of $\mathcal{H}_P$ in $H^1(P)$ and then let $\widehat{\mathcal{H}}_P\subset\mathscr{H}_P$ be the subspace of $f$ such that
\begin{equation}\label{DefofClass}
    \langle h,f\rangle_{H^1(P)}=P(h,f)+P(\nabla h,\nabla f)\leq C\Vert h\Vert_{L^2(P)},\quad \forall h\in\mathscr{H}_P,
\end{equation}
 for some $C>0$. Since $\mathscr{H}_P\supset \mathcal{H}_P$ is dense in $L^2(P)$, $\langle \cdot,f\rangle_{H^1(P)}$ extends to a bounded linear functional on $L^2(P)$. By the Riesz representation theorem \cite[Theorem II.4]{reed1972methods}, $\langle\cdot,f\rangle_{H^1(P)}$ can be identified with a unique $\eta\in L^2(P)$ such that
 \begin{equation}\label{DefofOpe}
     P(h,f)+P(\nabla h,\nabla f)=\langle h,f\rangle_{H^1(P)}=\langle h,\eta\rangle_{L^2(P)}=P(h,\eta),\quad \forall h\in \mathscr{H}_P.
 \end{equation}
 Hereby, $\widehat{\mathcal{L}}_P f$ is defined as $f-\eta$, which is consistent with the property of the initial pair $(\mathcal{L}_P,\mathcal{H}_P)$. Recall that for $f,h\in\mathcal{H}_P$, there is $P(h,f)+P(\nabla h,\nabla f)=P(h,f-\mathcal{L}_P f)$
by the second condition of extendability.
The \nameref{FriExt} further ensures that $\widehat{\mathcal{L}}_P$ is self-adjoint on $L^2(P)$ with $D(\widehat{\mathcal{L}}_P)=\widehat{\mathcal{H}}_P$ and maintains the property
\begin{equation}\label{FriPro}
    P(\widehat{\mathcal{L}}_P f, h)=-P(\nabla f,\nabla h),\quad f,h\in \widehat{\mathcal{H}}_P\subset H^1(P).
\end{equation}

At this juncture, the reader is cautioned that the extension $(\widehat{\mathcal{L}}_P,\widehat{\mathcal{H}}_P)$ can not be defined independently of the choice of initial domain $\mathcal{H}_P$, hence the term $\widehat{\mathcal{L}}_P$ must be used when $\mathcal{H}_P$ is already explicitly assigned. In such a case, we agree that the domain of $\mathcal{L}_P$ is $\mathcal{H}_P$, the domain of $\widehat{\mathcal{L}}_P$ is $\widehat{\mathcal{H}}_P$, and discuss their kernels or images without explicitly naming their domain.

The essence of Friedrichs extension is self-adjointness. For a closed operator $A$, there is $\ker A=(\im A^*)^\perp$ by proposition \ref{UnbOpe}. If $A$ is self-adjoint, i.e., $A=A^*$, we have $\ker A=(\im A)^\perp$. All aforementioned extensions, $\overline{\mathcal{L}}_P$ and $\mathcal{L}_P^*$, remain symmetric, but only the Friedrichs extension $\widehat{\mathcal{L}}_P$ is always self-adjoint and thus satisfies $\ker\widehat{\mathcal{L}}_P=(\im\widehat{\mathcal{L}}_P)^\perp$. Note that this property allows us to study the space, $\im\widehat{\mathcal{L}}_P$, which is unwieldy, by simply studying $\ker\widehat{\mathcal{L}}_P$, which is much easier to characterize. The following proposition captures this essence.

\begin{proposition}\label{SPchar} Suppose $\mathcal{L}_P$ is extendable on $\mathcal{H}_P$.
\begin{enumerate}
     \item $(\widehat{\mathcal{L}}_P,\widehat{\mathcal{H}}_P)$ is a Stein pair if and only if $\ker\widehat{\mathcal{L}}_P\supset \mathbb{R}$.
    \item $\im\widehat{\mathcal{L}}_P$ is dense in $L^2_0(P)$ if and only if $\ker\widehat{\mathcal{L}}_P=\mathbb{R}$.
\end{enumerate}
\end{proposition}
\begin{proof}
This is straightforward from the property $\ker\widehat{\mathcal{L}}_P=(\im\widehat{\mathcal{L}}_P)^\perp$.
\end{proof}

\begin{remark} Here $\mathbb{R}$ represents the subspace of all constant functions. Note that $1$ must be in $\ker\widehat{\mathcal{L}}_P$ if $1\in\widehat{\mathcal{H}}_P$, thus $\ker\widehat{\mathcal{L}}_P\supset\mathbb{R}$ is equivalent to $1\in\widehat{\mathcal{H}}_P$.
\end{remark}

 If there exists some $f\in\mathcal{H}_P$ with $P(\mathcal{L}_P f)\neq 0$, then $1$ must not be in $\widehat{\mathcal{H}}_P$, as well as any constant functions, since $P(\mathcal{L}_P f,1)=-P(\nabla f,\nabla 1 )=0$. In practice, we will certainly not choose such an $\mathcal{H}_P$ as our initial domain. However, on the other hand, though only an initial domain $\mathcal{H}_P$ on which $P(\mathcal{L}_P f)=0$ will be selected, it is typically hard to know whether $1\in\widehat{\mathcal{H}}_P$ holds true, unless $1$ is already in the initial domain $\mathcal{H}_P$.  Therefore, we may as well choose a $\mathcal{H}_P$ containing constant functions beforehand so as to let  $\ker\widehat{\mathcal{L}}_P\supset\mathbb{R}$.

In addition to the above discussion on  $\ker\widehat{\mathcal{L}}_P\supset \mathbb{R}$, a mild regularity condition will ensure that the converse statement also holds, namely, that constant functions are the only ones in $\ker\widehat{\mathcal{L}}_P$, i.e., $\ker\widehat{\mathcal{L}}_P\subset\mathbb{R}$. Surprisingly, this is independent of the choice of $\mathcal{H}_P$. This mild regularity condition is formalized in the following theorem.

\begin{theorem}\label{Kerconst} If $e^\phi\in L^1_{\loc}(M)$, then $\ker\widehat{\mathcal{L}}_P\subset\mathbb{R}$.
\end{theorem}
\begin{proof} Please see Appendix \ref{PrfKerconst}.
\end{proof}

The reader should note that theorem \ref{Kerconst} imposes a relatively weak condition, i.e., $e^\phi\in L^1_{\loc}(M)$, that is satisfied by all continuous densities, as they are locally bounded below by positive constants. At this juncture, from proposition \ref{SPchar} and theorem \ref{Kerconst}, we can conclude that $\im\widehat{\mathcal{L}}_P$ is dense in $L^2_0(P)$ with both $\ker\widehat{\mathcal{L}}_P\supset\mathbb{R}$ and $\ker\widehat{\mathcal{L}}_P\subset\mathbb{R}$ being satisfied under the aforementioned mild regularity condition. 

In summary, in this section we have presented a general recipe (framework) to extend the Stein pair for distributional comparisons to be applicable to a larger class of distributions. We are now ready to study the characterization scope of the class of distributions that can be distinguished using this extended Stein pair.

\subsection{Characterization Scope of the Extended Stein Pair
}

In this section we will study the scope of the class of distributions that can be discriminated using the Stein method and present novel techniques to substantially increase this scope. In achieving this increased scope, we will adaptively determine the stein pair along with other appropriate regularity conditions on the manifolds and the target distributions respectively.

Without any further assumptions imposed, $\im\widehat{\mathcal{L}}_P$ is only known  to be dense in $L^2_0(P)$. In fact, a dense subspace in $L^2_0(P)$ is already fairly rich so that $(\widehat{\mathcal{L}}_P,\widehat{\mathcal{H}}_P)$ will distinguish $P$ between a considerable scope of distributions.
However, recall that the elements in $L^2(P)$ are not really functions, but equivalence classes of functions distinct up to a $P$-null set. In order for the integration, $E[\widehat{\mathcal{L}}_Pf(X)]$, to make sense, the characterization only 
encapsulates situations where,
$X$ with $Q_{X}\ll P$, or equivalently, $Q_{X}\ll v$. As for the characterization of weak convergence, different regularity conditions lead to different characterization scope of distributions. The following theorem captures the characterization scope of the Stein pair under different regularity conditions.

\begin{theorem}\label{CC} Suppose $\mathcal{L}_P$ is extendable on $\mathcal{H}_P$ and $\ker\widehat{\mathcal{L}}_P=\mathbb{R}$. Suppose $X$ and $X_n$ are $M$-valued r.v. with $Q_X,Q_{X_n}\ll P\sim\lambda$.

\begin{enumerate}[(C1)]

\item If $\frac{d Q_X}{d P}\in L^2(P)$, then
    $$Q_X= P \iff E[\widehat{\mathcal{L}}_Pf(X)]=0,\ \forall f\in\widehat{\mathcal{H}}_P; $$

\item If $\sup \Vert \frac{d Q_{X_n}}{d P} \Vert_{L^2{(P)}}<+\infty$, then
$$\frac{d Q_{X_n}}{d P}\xrightarrow{w.} 1\ \text{\em in } L^2(P) \iff  E[\widehat{\mathcal{L}}_Pf(X_n)]\to 0,\ \forall f\in\widehat{\mathcal{H}}_P; $$

\item If $\im\widehat{\mathcal{L}}_P=L^2_0(P)$, that is, $\widehat{\mathcal{L}}_P$ is surjective, then 
$$\frac{d Q_{X_n}}{d P}\xrightarrow{w.} 1\ \text{\em in } L^2(P) \iff  E[\widehat{\mathcal{L}}_Pf(X_n)]\to 0,\ \forall f\in\widehat{\mathcal{H}}_P. $$
\end{enumerate}
\end{theorem}
\begin{proof}
Please see Appendix \ref{PrfCC}.
\end{proof}

\begin{remark} Note that when using the expression, $E[\widehat{\mathcal{L}}_P f(X)]=0$, it is implicitly understood that $\widehat{\mathcal{L}}_P$ is $Q_X$-integrable, i.e., $E\big|\widehat{\mathcal{L}}_P f(X)\big|<+\infty$. Further, when using the expression, $ E[\widehat{\mathcal{L}}_Pf(X_n)]\to 0$, it is understood that $\widehat{\mathcal{L}}_P f$ is $Q_{X_n}$-integrable except at a finite number of $n$.
\end{remark}

\begin{remark} The term $\frac{d Q_{X_n}}{d P}\xrightarrow{w.} 1$ in $L^2(P)$ stands for the convergence in weak topology of $L^2(P)$ as a Hilbert space. Such a convergence implies the weak convergence $Q_{X_n}\xrightarrow{w.} P$ in distribution sense.
\end{remark}

\subsubsection{Essential Self-adjointness}\label{ESA}

Regardless of the specific choice of initial domain $\mathcal{H}_P$, the pair $(\widehat{\mathcal{L}}_P,\widehat{\mathcal{H}}_P)$ based on Friedrichs extension is typically highly implicit, since $\widehat{\mathcal{L}}_P$ is defined in weak sense and $\widehat{\mathcal{H}}_P$ relies on the completion under $H^1$ norm. As opposed to the extended pair $(\widehat{\mathcal{L}}_P,\widehat{\mathcal{H}}_P)$, we may wish to study the characterization scope of $(\mathcal{L}_P,\mathcal{H}_P)$ itself, which is typically explicitly specified in most cases.
Fortunately, the framework we present in this work provides us a means to this end.

\begin{theorem}\label{ESAchar} Suppose $\mathcal{L}_P$ is extendable on $\mathcal{H}_P$. If $\mathcal{L}_P$ is essentially self-adjoint on $\mathcal{H}_P$, then $\im\mathcal{L}_P$ is dense in $\im\widehat{\mathcal{L}}_P$. In such a case, $(\widehat{\mathcal{L}}_P,\widehat{\mathcal{H}}_P)$ is a Stein pair if 
$(\mathcal{L}_P,\mathcal{H}_P)$ is a Stein pair, even when $1\notin\mathcal{H}_P$.
\end{theorem}
\begin{proof}
When $\mathcal{L}_P$ is essentially self-adjoint, we have $\overline{\mathcal{L}}_P=\widehat{\mathcal{L}}_P$, that is, for all $f\in D(\widehat{\mathcal{L}}_P)=\widehat{\mathcal{H}}_P$, there exists $f_n\in \mathcal{H}_P $ such that $f_n\to f,\  \mathcal{L}_P f\to\widehat{\mathcal{L}}_P f$ in $L^2(M;P)$. This implies that $\im\mathcal{L}_P$ is dense in $\im\widehat{\mathcal{L}}_P$, thus dense in $L^2_0(P)$. It also implies that $P(\widehat{\mathcal{L}}_P f)=0$ for all $f\in\widehat{\mathcal{H}}_P$ if $(\mathcal{L}_P,\mathcal{H}_P)$ is a Stein pair.
\end{proof}

\begin{remark} In fact, the converse of theorem \ref{ESAchar} is also true if any of the three statements in theorem \ref{Sur} hold, which can be easily established using \cite[Theorem X.26]{reed1975ii}. However, such a result will not be used in this work, since we can establish stronger results when $\im\widehat{\mathcal{L}}_P$ is surjective.
\end{remark}

If $\mathcal{L}_P$ is essentially self-adjoint on $\mathcal{H}_P$ and $\ker\widehat{\mathcal{L}}_P=\mathbb{R}$,  $(\widehat{\mathcal{L}}_P,\widehat{\mathcal{H}}_P)$ in Theorem \ref{CC} can be replaced with $(\mathcal{L}_P,\mathcal{H}_P)$. 
It should be noted that essential self-adjointness property being satisfied is highly dependent on the specific structure of $M$. Following theorem provides a non-trivial case in point.

\begin{theorem}\label{ESAcpl} Suppose $M$ is complete without boundary and $e^{-\phi}\in C^L_{\loc}(M)$, then $\mathcal{L}_P$ is essentially self-adjoint on $C^\infty_c(M)$.
\end{theorem}
\begin{proof}
Please see Appendix \ref{PrfESAcpl}.
\end{proof}

Under the assumption of theorem \ref{ESAcpl}, we have an application in this case.

 \begin{corollary}\label{ESAapp} Under the assumptions made in theorem \ref{ESAcpl}, for $M$-valued r.v. $X_n$ with $Q_{X_n}\ll P$ and $\sup\Vert\frac{d Q_{X_n}}{d P}\Vert_{L^2(P)}<+\infty$,
$$ Q_{X_n}\xrightarrow{w.} P \iff  E[\mathcal{L}_Pf(X_n)]\to 0,\ \forall f\in C^\infty_c(M). $$
\end{corollary}
\begin{proof} Directly applies the (C2) in theorem \ref{CC}.
\end{proof}

 We would like to emphasize that 
the conditions we impose on $P$ are weaker than those in related published works \cite{barp2018riemann,le2020diffusion,thompson2020,xu2021interpretable}, in that it is only \emph{locally Lipschitz continuous}, which allows us to analyze the intrinsic distributions, as illustrated in the following example.
This is a significant distinction that merits attention as it motivates our method of Stein pair extension on Riemannian manifolds.

\begin{example}\label{Intrinsic} Consider the family of distributions  $\frac{d P}{d v}\propto \exp(-d(x,\mu)^\alpha)$ for some $\alpha\geq 1$ and $\mu\in M$, where $M$ is a complete Riemannian without boundary. For the case $\alpha=2$, it is the intrinsic Normal law \cite{pennec2006intrinsic} deduced through maximizing the entropy under specific conditions on the second order moment. Such function $\exp(-d(x,\mu)^\alpha)$ is clearly locally Lipschitz continuous. Therefore, corollary \ref{ESAapp} applies to this case.
\end{example}

\subsubsection{Weighted Poincaré Inequality} \label{WPI}

In this section, we will first establish the if and only if conditions for surjectivity of $\widehat{\mathcal{L}}_P$, which facilitates testing for surjectivity. Following this, we will use the surjectivity of $\widehat{\mathcal{L}}_P$ to establish improved scope of characterization of the Stein pair. 

\begin{theorem}\label{Sur} Suppose $\mathcal{L}_P$ is extendable on $\mathcal{H}_P$ and $\ker\widehat{\mathcal{L}}_P=\mathbb{R}$, then following statements are equivalent:
\begin{enumerate}
    \item $\im\widehat{\mathcal{L}}_P=L^2_0(P)$.
    \item $\exists C_1>0$ s.t. $\Vert f-P(f)\Vert_{L^2(P)}\leq C_1\Vert\widehat{\mathcal{L}}_P f\Vert_{L^2(P)} $, for all $f\in\widehat{\mathcal{H}}_P$.
    \item $\exists C_2>0$ s.t. $\Vert f-P(f)\Vert_{L^2(P)}\leq C_2 \Vert \nabla f\Vert_{L^2(P)} $, for all $f\in\mathcal{H}_P$.
\end{enumerate}
\end{theorem}
\begin{proof} Please see Appendix \ref{PrfSur}.
\end{proof}

The inequality in the third statement of theorem \ref{Sur} is known as the \emph{weighted Poincaré inequality} (WPI). The smallest feasible value of constant $C_2$ is exactly the first non-zero eigenvalue of $\widehat{\mathcal{L}}_P$ as a self-adjoint operator. Therefore, the WPI holds if and only if $0$ is isolated in the spectrum of $\widehat{\mathcal{L}}_P$. The existing literature \cite{bandara2018eigenvalue,colbois2013eigenvalues,du2021estimates,setti1998eigenvalue,wang2012eigenvalue,wang2016lower} on this inequality and the estimation of the first non-zero eigenvalue is rather immense and entangled with a variety of fields of Science and Engineering. Since our work does not rely on the specific estimation technique used for the estimation of $C_2$ -- which involves enormous effort to summarize here  and is out of the scope of this work --  we will simply focus on several well known cases (for subsequent analysis) where WPI holds.

\begin{proposition}\label{WPIcase} The weighted Poincaré inequality holds on $H^1(P)$ if
\begin{enumerate}[(C1)]
    \item $M=\mathbb{R}^n$ and $\phi$ is convex. 
    \item $M$ is complete without boundary and $\ric+\text{\em Hess}(\phi)\geq \kappa g$ for some $\kappa>0$.
    \item $M$ is compact, $\phi$ is continuous on $M$.
\end{enumerate}
\end{proposition}

\begin{proof} See  \cite{bobkov1999isoperimetric} for (C1) and \cite{cheng2017eigenvalues} for (C2), and (C3) is straightforward from the \nameref{PI} since $P\asymp v$ in this case.
\end{proof}

When WPI holds on $\mathcal{H}_P$, $\im\widehat{\mathcal{L}}_P$ will be able to capture all functions in $L^2_0(P)$, which is more than enough to characterize weak convergence. However, as mentioned previously, elements in $\im\widehat{\mathcal{L}}_P$ are defined weakly, instead of pointwise. Therefore, even when the WPI holds on $\mathcal{H}_P$, the pair $(\widehat{\mathcal{L}}_P,\widehat{\mathcal{H}}_P)$ still only applies to distributions absolutely continuous w.r.t $P$. In other words, $\im\widehat{\mathcal{L}}_P$ is excessively large, among which some functions are not appropriate for the characterization of weak convergence of arbitrary distributions.

Recall that it is $C_b(M)$ that is used to characterize the weak convergence. A direct argument involving the partition of unity and Stone–Weierstrass theorem \cite[Theorem IV.9]{reed1972methods} would illustrate that $C^\infty_b(M)$ is dense in $C_b(M)$ under the uniform norm. Therefore, to drop the condition $Q_X\ll P$, we need a function class $\mathcal{C}\subset C^2(M)$ on which $\mathcal{L}_P$ is defined strongly and $C^\infty_b(M)\subset \im\mathcal{L}_P|_\mathcal{C}\subset C_b(M)$.

To establish relevant results, we invoke a proposition first.

\begin{proposition}\label{RglPDE} Suppose $\mathcal{L}_P$ is extendable on $\mathcal{H}_P$, $\ker\widehat{\mathcal{L}}_P=\mathbb{R}$, the WPI holds on $\mathcal{H}_P$ and $e^{-\phi}\in C^{\lfloor \frac{n}{2}\rfloor+k+3}$ for some $k\geq 0$. Then for each $h\in C^\infty_b(M)$ with $P(h)=0$, there exists $f_h\in \widehat{\mathcal{H}}_P\cap C^{k+2}(M)$ such that $\mathcal{L}_P f=h$. 
\end{proposition}
\begin{proof} This is straightforward by the regularity results of the solutions to elliptic equations \cite[\S 6.3]{evans2022partial} and Sobolev embedding theorem \cite[Theorem 4.12]{adams2003sobolev}.
\end{proof}

As mentioned previously, the function class $\widehat{\mathcal{H}}_P$ as well as  $\widehat{\mathcal{H}}_P\cap C^{k+2}(M)$ could be highly implicit. Therefore, we could adopt different function classes as the manifold $M$ varies. The following theorem \ref{SCC} captures the characterization scope of the Stein pair under different regularity conditions imposed on $M$.

\begin{theorem}\label{SCC} Under the assumptions of Proposition \ref{RglPDE}, for a sequence of $M$-valued r.v.s $X_n$, not necessarily with $Q_{X_n}\ll P$, 
$$ Q_{X_n}\xrightarrow{w.} P \iff E[\mathcal{L}_P f(X_n)]\to 0,\ \forall f\in \mathcal{D}^k(M;P). $$
The function class $\mathcal{D}^k_P(M)$ can be prescribed as follows:
\begin{enumerate}[(C1)]

\item For a general incomplete $M$,
    $$ \mathcal{D}^k_P(M):=\left\{f\in C^{k+2}(M)\cap L^2(P): \mathcal{L}_P f\in C_b(M), P(\mathcal{L}_P f)=0 \right\}.$$

\item If $M$ is complete without boundary,
       $$ \mathcal{D}^k_P(M):=\left\{f\in C^{k+2}(M)\cap L^2(P): \mathcal{L}_P f\in C_b(M) \right\}. $$

\item If $M$ is compact with boundary and $\vec{n}$ is the normal vector field of $\partial M$, $$\mathcal{D}^k(M):=\left\{f\in C^{k+2}(M): g(\vec{n},\nabla f)=0 \right\}. $$
 If $\partial M=\emptyset$, then, $\mathcal{D}^k(M)=C^{k+2}(M)$. This class is independent of $P$.
\end{enumerate}
\end{theorem}
\begin{proof}
Please see Appendix \ref{PrfSCC}.
\end{proof}

\subsection{Applications}\label{applications}
We are now ready to present several (non-numeric) examples that illustrate the application of the main results presented in the previous section.

\begin{example}[Cut Locus and Incomplete Manifolds] \label{Cut Locus}
Suppose $M$ is a compact manifold without boundary. Consider the distribution 
$$ \frac{d P}{d v}\propto \exp(-\Log_\mu(x)^T\Gamma_\mu \Log_\mu(x))$$
for some  $\mu\in M$ and positive definite tensor in $\Gamma_\mu\in T^{0,2}_\mu M$. Such a density is smooth in $\mathscr{N}_\mu:=M\setminus\mathscr{C}_\mu$ but not even continuous on the cut locus $\mathscr{C}_\mu$. Therefore, we adopt $\mathscr{N}_\mu$ as the support of $P$, which is an incomplete manifold. Note that $P\asymp v$ on $M$, thus WPI holds on $C^\infty_c(\mathscr{N}_\mu)\oplus \mathbb{R}$. Therefore, (C1) in theorem \ref{SCC} applies to this case. For any sequence of $M$-valued r.v. $X_n$ with $\prob\{X_n\in \mathscr{C}\}=0$,
$$ Q_{X_n}\xrightarrow{w.} P \iff E[\mathcal{L}_P f(X_n)]\to 0,\ \forall f\in \mathcal{D}^\infty(\mathscr{N}_\mu;P). $$
\end{example}

\begin{example}[Truncated Distributions]
\label{Truncation}
In many applications, to secure the uniqueness of the Fréchet mean of an $M$-valued r.v.,  it is necessary to truncate the distribution to lie within a convexity ball $B(\mu,r)$ for some $\mu\in M$ and $r>0$, that is, a ball on which $d(\cdot,\mu)$ is geodesically convex and between any two points within this region, there is a unique geodesic. Note that a closed and bounded subset on a complete manifold is compact by Hopf-Rinow theorem \cite[Theorem 5.7.1]{petersen2016riemannian}, so it is very common to assume that $B(\mu,r)$ is compact and has smooth boundary. In such cases, (C3) in theorem \ref{SCC} applies.
\end{example}

\begin{example}[Kernel Stein Discrepancy (KSD)]\label{KSD} In this example, we highlight a very special case with significant practicality, where the initial function class $\mathcal{H}_k$ is a \emph{reproducing kernel Hilbert space}(RKHS) reproduced by a smooth kernel function $k:M\times M\to\mathbb{R}$. For readers not familiar with such notions, we refer to \cite{berlinet2011reproducing}.

Let $k_P:=\mathcal{L}'_P\mathcal{L}_P k(x,x')$ represent the function obtained by letting $\mathcal{L}_P$ act on $x$ and $x'$ in order. A classical argument will show that 
$$ \sup \big\{ E[\mathcal{L}_P f(X) ]: f\in\mathcal{H}_k, \Vert f\Vert_{\mathcal{H}_k}\leq 1 \big\}= E k_P(X,X'), $$
where $X$ and $X'$ are $Q$-distributed and mutually independent $M$-valued random variables.
This value of the above equation is denoted by $\ksd(Q,P)$,  called \emph{kernel stein discrepancy}(KSD). Since the right hand side of above equation only involves one integral, the $\ksd$ is very easy to compute in applications and even easier when $Q$ is discrete. If $(\mathcal{L}_P,\mathcal{H}_k)$ is a Stein pair, we can similarly define the notion that a $\ksd$ characterizes $P$ if $\ksd(Q,P)=0\iff Q=P $, characterizes the weak convergence if $\ksd(Q_n,P)\to 0\iff Q_n\xrightarrow{w.} P$. In general, both of them are not true, but authors in \cite{barp2018riemann} discovered a special case where the $\ksd$ characterizes the weak convergence. Our framework provides a different path to establish this result as elaborated upon below.

Suppose $M$ is compact without boundary and $e^{-\phi}\in C^L(M)$. It is known that $H^s(M)$ for $s>\frac{n}{2}$ is a RKHS reproduced by some kernel function $k_s$ \cite{de2021reproducing}. We consider some $s>\frac{n}{2}+2$, so, by the Sobolev embedding theorem \cite[Proposition 3.3]{taylor2000PDEI}, $H^s(M)\hookrightarrow C^2_b(M)$. Recall that $C^\infty(M)\subset H^s(M)$. We now apply the theory in \S \ref{ESA} to conclude that for a sequence of distribution $Q_n$ on $M$ such that $Q_n\ll P$ and $\sup \Vert\frac{d Q_n}{d P}\Vert_{L^2(P)}<+\infty$,
$$ \ksd(Q_n,P)\to 0 \iff Q_n\xrightarrow{w.} P. $$
Furthermore, if $e^{-\phi}\in C^{s+2}(M)$, which is the case considered in \cite{barp2018riemann}, then we apply the theory in \S \ref{WPI}  so that the condition $Q_n\ll P$ and $\sup\Vert \frac{d Q_{X_n}}{d P} \Vert_{L^2(P)}<+\infty$ can be dropped.
\end{example}

\begin{example}[Hyperbolic space]\label{Hyperbolic} The $n$-dimensional hyperbolic space $\mathbb{H}^n$ is a well-known complete Riemannian manifold with constant sectional curvature $-1$. Consider the density function $\frac{d P}{d v}\propto \exp\left(-\frac{ d(x,\mu)^2}{\sigma^2}\right)$ for some $\mu\in\mathbb{H}^n$ and $\sigma>0$,  where $d$ is the Riemannian distance on $\mathbb{H}^n$. There is no cut locus on $\mathbb{H}^n$, hence $d^2(\cdot,\mu)$ is smooth globally. For detailed introduction on the structure of such space, we refer to readers to \cite{pennec2018barycentric}. 

Now we show that the WPI holds when $\sigma<(n-1)^{-\frac{1}{2}} $. By (C2) in proposition \ref{WPIcase}, it suffices to show that $\ric+\sigma^{-2}\hess(d(\cdot,\mu)^2)\geq \kappa g$ for some $\kappa>0$. The Ricci curvature and Hessian for $\mathbb{H}^n$ are known to be,  $\ric= -(n-1) g$ \cite[lemma 8.10]{lee2006riemannian} and $ \hess( d(\cdot,\mu)^2 )\geq g$ \cite[\S 2.3]{pennec2018barycentric} respectively.  Setting $\kappa= \sigma^2-n+1$ when $\sigma<(n-1)^{-\frac{1}{2}} $ makes 
(C2) in theorem \ref{SCC} applicable to this case.
\end{example}

\section{Conclusion}\label{Conc}

In this paper, we presented a novel framework for improving the characterization scope of the Stein pair on general Riemannian manifolds.  This improvement was achieved using the Friedrichs extension
applied to self-adjoint unbounded operators. The key feature of this framework is that it allows for analyzing the characterization scope of the distributions on Riemannian manifolds that can be compared. The stronger the imposed regularity conditions on the manifold and the target distributions, the stronger is the characterization scope of the resulting Stein pair. We presented several examples illustrating the application of our theory to a variety of unconventional situations including  intrinsically defined non-smooth distributions, truncated distributions on Riemannian manifolds and distributions on incomplete Riemannian manifolds.
Our future work will focus on developing applications of the theory presented here to a variety of manifold-valued data analysis problems encountered in imaging sciences.

\paragraph{Acknowledgement.} 
This research was in part funded by the NSF grant IIS 1724174 and the NIH NINDS and NIA grant R01NS121099 to Vemuri.

\begin{appendix} \label{Appendix}
\section{Proofs of Theorems}

In this appendix, we provide the proofs of all the theorems that are original to our work and were presented in \S\ref{MainResults}. For the sake of convenience, we have hyperlinked the section titles to the theorems statements in text.

\subsection{Proof of Theorem \ref{Kerconst}}\label{PrfKerconst}
\begin{proof}
It suffices show that for every function $f\in\widehat{\mathcal{H}}_P$ if $\widehat{\mathcal{L}}_P f=0$, then $f$ is a constant function. For such an $f$, note that $P(\nabla f,\nabla f)=-P(\widehat{\mathcal{L}}_P f, f)=0$ by  property \ref{FriPro}, hence, $\nabla f=0$ in $\widehat{\mathcal{H}}_P$. Recall that by the procedure of Friedrichs extension, there exists a sequence of $h_n\in\mathcal{H}_P\subset H^2(P)$ such that
$$ h_n\to f\ \text{in}\ L^2(P),\quad \int_M|\nabla h_n|^2 d P=  P(\nabla h_n,\nabla h_n)\to 0.  $$
Since $\mathcal{C}^{2,2}$ is dense in $H^2(P)$, without loss of generality, we may further assume that each $h_n$ is $C^\infty$. For each precompact open neighborhood $U$ with smooth boundary, $e^\phi$ is integrable on $U$.  Let $(h_n)_U=[v(U)]^{-1}\int_U h_n d v$ be the average of $h_n$ over $U$. By the \nameref{PI} with $p=1$ and the Hölder inequality \cite[equation 19.3]{billingsley2008probability}, we have, 
\begin{eqnarray*}
\int_U | h_n-(h_n)_U|d v
&\leq& \int_U |\nabla h_n| d v=\int_U e^{\frac{\phi}{2}}|\nabla h_n| e^{-\frac{\phi}{2}}d v\\
&\leq& \left[\int_U e^\phi d v\right]^{\frac{1}{2}} \cdot\left[\int_U |\nabla h_n|^2 e^{-\phi}d v\right]^{\frac{1}{2}}.\\
&\leq & \left[ C_\phi \int_U e^\phi d v\right]^{\frac{1}{2}} \cdot \sqrt{P(\nabla h_n,\nabla h_n)},
\end{eqnarray*}
where $C_\phi$ is the normalizing constant of $e^{-\phi}$.
Note that $P(\nabla h_n,\nabla h_n)\to 0$ as $n\to\infty$, which implies that $h_n$ converges to some constant in $L^1(U)$. Therefore, $f$ is constant a.e. on $U$. Since the choice of $U$ is arbitrary, we have that $f$ is locally constant everywhere. Therefore, $f$ is  globally constant since $M$ is connected. This completes the proof.
\end{proof}

\subsection{Proof of Theorem \ref{CC}}\label{PrfCC}

\begin{proof}

(C1) Since $Q_X(M)=P(M)=1$, the integration under $P$ and $Q$ agree on $\im\widehat{\mathcal{L}}_P\oplus\mathbb{R}$, which is a dense subspace of $L^2(P)$. Let $\rho=\frac{d Q_X}{d P}$, then $P(f)=Q(f)=P(f\rho)$, for $f\in \im\widehat{\mathcal{L}}_P\oplus\mathbb{R} $. Since $\rho\in L^2(P)$, $\rho-1\in L^2(P)$, thus $P(f,\rho-1)=0$ for $f\in \im\widehat{\mathcal{L}}_P\oplus\mathbb{R} $, which implies $\rho=1$.

\medskip
(C2) Suppose $\rho_n:=\frac{d Q_{X_n}}{d P}$ satisfy $\sup \Vert\rho_n\Vert_{L^2(P)}<+\infty$. Note that $ E[\widehat{\mathcal{L}}_P f(X_n)]\to 0$ for all $f\in\widehat{\mathcal{H}}_P$ is equivalent to that $\langle h,\rho_n\rangle_{L^2(P)}\to 0$ for all $h\in \im\widehat{\mathcal{L}}_P\oplus\mathbb{R}$. This is further equivalent to  $\langle h,\rho_n\rangle_{L^2(P)}\to 0$ for all $h\in L^2(P)$, since $\im\widehat{\mathcal{L}}_P\oplus\mathbb{R}$ is dense in $L^2_0(P)$ and $\sup \Vert\rho_n\Vert_{L^2(P)}<+\infty$. This is exactly the weak convergence in $L^2(P)$.

\medskip
(C3) Suppose $\im\widehat{\mathcal{L}}_P=L^2_0(P)$. Let 
$$ A_{m,n}:=\left\{h\in L^2_0(P): E[|h(X_k)|]\leq 1/m,\ \forall k\geq n   \right\}. $$
Clearly, $\bigcup A_{m,n}=L^2_0(P)$. By Fatou's lemma \cite[Theorem 16.3]{billingsley2008probability}, each $A_{m,n}$ is closed. By Baire's category theorem \cite[Theorem III.8]{reed1972methods}, there exists a closed ball $\overline{B}_r(h_0)\subset A_{m_0,n_0}$ for some $(m_0,n_0)$. For each $h\in L^2_0(P)$, note that $h_0,h_0+r h/\Vert h\Vert \in \overline{B}_r(h_0)\subset A_{m_0,n_0}$, thus 
$$ E|h(X_k)|\leq r^{-1}\Vert h\Vert\cdot [E|h_0(X_k)|+E|(h_0+r h/\Vert h\Vert)(X)|]\leq 2 (r m_0)^{-1}\Vert h\Vert,$$
for $k\geq n_0$. Therefore, $\rho_k$ satisfy the condition in (2) for a large enough $k$.
\end{proof}

\subsection{Proof of Theorem \ref{ESAcpl}}\label{PrfESAcpl}

\begin{proof}
For this proof, we use a classical proof technique introduced in $\cite[\S X.1]{reed1975ii}$. This proof technique was also used for the unweighted Laplacian $\Delta$  in \cite{strichartz1983analysis}. In this proof, we are faced with the weighted Laplacian case.

Since $\mathcal{L}_P$ is a symmetric operator semi-bounded above by $0$, it suffices to show that $\dim(\ker(I-\mathcal{L}_P^*))=0$. Suppose $u\in\ker(\lambda I-\mathcal{L}_P^*)$, i.e. , $ P(h- \mathcal{L}_P h,u)=0 $ for all $h\in C^\infty_c(M)$. We will now show that $u=0$.

In each local coordinate chart, $u$ is a very weak solution of an elliptic equation with Lipschitz coefficients. Thus, an application of the regularity results stated in \cite{zhang2012regularity} leads to the conclusion that $u\in H^2_{\loc}(P)$. For $h\in C^\infty_c(M)$, applying the \nameref{DivThm} leads to
$$ P(h,u)=P(\mathcal{L}_P h, u)= -P(\nabla h,\nabla u), $$
which actually implies that $P(\psi,u)=-P(\nabla \psi,\nabla u)$ for all $\psi\in H^1_c(P)$. 

Let $\phi\in C^\infty(\mathbb{R})$ be a function such that $\phi|_{[-1,1]}=1$, $\supp(\phi)= [-2,2]$ and $|\phi'|\leq 2$. Consider a point $o\in M$, we define $\phi_n$ on $M$ as $\phi_n(x)=\phi\left(n^{-1}d(x,o)\right)$. Then $\phi_n\in C_c(M)$ satisfy $\Vert\nabla\phi_n\Vert_\infty\leq 2/n$ and $\phi_n\to 1$.

Since $u\in H^2_{\loc}(P)$ and $\phi_n\in C_c(M)$ is Lipschitz continuous, $\phi_n^2 u\in H^1(M)$. Therefore, we have
$$ 0\leq P(\phi_n^2 u, u) = -P(\nabla(\phi_n^2 u),\nabla u)=-P(\phi_n^2\nabla u,\nabla u)-P(2 \phi_n u \nabla\phi_n,\nabla u ), $$
which further implies that
$$\Vert \phi_n\nabla u\Vert^2_{L^2(P)}\leq \Vert\nabla\phi_n\Vert_{\infty}\cdot\Vert u\Vert_{L^2(P)}\cdot\Vert\phi_n\nabla u\Vert_{L^2(P)}. $$
Therefore, $\Vert \phi_n\nabla u\Vert_{L^2(P)}\leq 2 n^{-1}\Vert u\Vert_{L^2(P)}\to 0$. Hence, $\nabla u=0$ and
 $P(\phi_n^2 u,u)=P(\nabla (\phi_n^2 u),\nabla u)=0 $
for all $n$. Therefore, $u=0$. This completes the proof.
\end{proof}

\subsection{Proof of Theorem \ref{Sur}}\label{PrfSur}

\begin{proof}
 We will first show that $\mathbf{1\Rightarrow 2}$. Let $\widehat{\mathcal{H}}_{P,0}:=\widehat{\mathcal{H}}_P\big/\mathbb{R}^n$ denote the centered subspace of $\widehat{\mathcal{H}}_P$, that is, all functions $f$ in $\widehat{\mathcal{H}}_P$ with $P(f)=0$. Consider the graph $\mathcal{G}$ of $\widehat{\mathcal{L}}_P$ on $\widehat{\mathcal{H}}_{P,0}$, i.e.,
$$ \mathcal{G}:=\left\{ (f,\widehat{\mathcal{L}}_P f)\in L^2(P)\times L^2(P):f\in \widehat{\mathcal{H}}_{P,0} \right\}, $$
which inherits the inner product on $L^2(P)\times L^2(P)$. $\mathcal{G}$ is closed since $\widehat{\mathcal{L}}_P$ is closed, and thus is a Hilbert space. Since $\im\widehat{\mathcal{L}}_P=L^2(P)$ and $\ker\widehat{\mathcal{L}}_P$ contains only constant functions, the projection $\pi:(f,\widehat{\mathcal{L}}_P f)\mapsto \widehat{\mathcal{L}}_Pf$ is a bounded linear bijection on $\mathcal{G}$. By inverse mapping theorem \cite[Theorem III.11]{reed1972methods}, the inverse of $\pi$ is bounded, from which we conclude $\mathbf{2}$. This proves the forward implication.

\medskip

We now proceed with the converse i.e.,  $\mathbf{2\Rightarrow 1}$. $\im\widehat{\mathcal{L}}_P$ is dense in $L^2(P)$, hence for each $\psi\in L^2(P)$ there exists a sequence $h_n\in \widehat{\mathcal{H}}_P$ such that $\widehat{\mathcal{L}}_P h_n\to\psi $ in $L^2(P)$ and $P(h_n)=0$. Statement $\mathbf{2}$ implies $h_n$ is Cauchy. As a result, $h_n$ is also Cauchy in $\mathscr{H}_P$ by property \ref{FriPro} and converges to some $h\in \mathscr{H}_P$. Note that
$$ P[\eta\cdot (h+\psi)]=\lim_{n\to\infty} P[\eta\cdot(h_n+\widehat{\mathcal{L}}_P h_n)]\leq \sup_{n\geq 1}\Vert h_n+\widehat{\mathcal{L}}_P h_n \Vert_{L^2(P)}\cdot \Vert \eta\Vert_{L^2(P)}, $$
for all $\eta\in\widehat{\mathcal{H}}_P$. Recall the definition \ref{DefofClass} and $\ref{DefofOpe}$ of $(\widehat{\mathcal{L}}_P,\widehat{\mathcal{H}}_P)$, we have that $h\in\widehat{\mathcal{H}}_P$ and $\widehat{\mathcal{L}}_P h=h+\psi-h=\psi$. This concludes the reverse implication.
\medskip

Next we prove that $\mathbf{3\Rightarrow 2}$. The statement that the inequality holds on $\mathcal{H}_P$ actually implies that it holds on the entire $\widehat{\mathcal{H}}_P$. For each $h\in\widehat{\mathcal{H}}_P$, application of Cauchy-Schwarz inequality leads to,
$$ \Vert h\Vert_{L^2(P)}^2\leq C^2 \Vert \nabla h\Vert^2_{L^2(P)}=C^2 P(\widehat{\mathcal{L}}_P h\cdot h)\leq C^2\Vert\widehat{\mathcal{L}}_P h \Vert_{L^2(P)}\cdot\Vert h\Vert_{L^2(P)}. $$ This completes the proof of $\mathbf{3\Rightarrow 2}$.
\medskip

We now show that $\mathbf{1,2\Rightarrow 3}$. We restrict $\widehat{\mathcal{L}}_P$ on $L_0^2(P)$. Then $\widehat{\mathcal{L}}_P$ is self-adjoint on $L^2_0(P)$ and  $D(\widehat{\mathcal{L}}_P)=\widehat{\mathcal{H}}_{P,0}$. The spectral theorem \cite[Theorem VIII.4]{reed1972methods} states that there exists a finite measure space $(\mathcal{S},\mu)$, a unitary operator $\mathcal{U}:L^2_0(P)\to L^2(\mathcal{S},\mu)$ and a real-valued function $f$ on $\mathcal{S}$ such that
\begin{itemize}
    \item $f$ is $\mu$-a.e. finite, that is, $|f|<+\infty$ $\mu$-a.e.
    \item $h\in \widehat{\mathcal{H}}_{P,0} $ if and only if $f\cdot\mathcal{U}h\in L^2(\mathcal{S},\mu)  $,
    \item $\mathcal{U}(\widehat{\mathcal{L}}_P h)=f\cdot\mathcal{U}h$ for $h\in\widehat{\mathcal{H}}_{P,0}$.
\end{itemize}
Let $\mathscr{F}:=\{\mathcal{U}h:h\in\widehat{\mathcal{H}}_{P,0}\}$, or equivalently, $ \{\psi\in L^2(\mathcal{S},\mu):f\cdot\psi\in L^2(\mathcal{S},\mu)\}$. Recall the properties of $\widehat{\mathcal{L}}_P$ and the statements $\mathbf{1}$ and $\mathbf{2}$, we have
\begin{enumerate}[(a)]
    \item $\mu(f\cdot \psi^2)\geq 0$ for $\psi\in \mathscr{F}$,
    \item $C\Vert f\psi \Vert_{L^2(\mathcal{S},\mu)}\geq \Vert\psi\Vert_{L^2(\mathcal{S},\mu)}$ for $\psi\in\mathscr{F}$,
    \item $\{f\psi:\psi\in\mathscr{F}\}=L^2(\mathcal{S},\mu)$.
\end{enumerate}

First we show $f<0$ is $\mu$-a.e. Suppose not, then, $\mu\{f\geq 0\}>0$. Since $\mu$ is finite, $I_{\{f\geq 0\}}\in L^2(\mathcal{S},\mu)$ and by property (c), there exists a $\psi_0\in\mathscr{F}$ such that, $f\psi_0=I_{\{f\geq 0\}}$ $\mu$-a.e. . Note that $f=0$ but $f\cdot\psi_0=1$ on $\{f=0\}$, thus $\mu\{f=0\}=0$. Similarly, we have $\psi_0>0$ $\mu$-a.e. on $\{f>0\}$. Therefore, 
$$ \mu(f\psi_0^2)=\mu(\psi_0 I_{\{f>0\}}+\psi_0\cdot I_{\{f=0\}} ) =\mu(\psi_0 I_{\{f>0\}})<0,$$ 
which  contradicts the property (a). Hence the result.
\medskip

Now we will show that $f\leq - C^{-1}$ $\mu$-a.e. Suppose not, then, $\mu\{f> -C^{-1}\}>0$, and again there exists $\psi_0\in\mathscr{F}$ with $f\psi_0=I_{\{f> - C^{-1}\}}$ $\mu$-a.e. Since $f\neq 0$ $\mu$-a.e. but $f\cdot\psi_0=0$ on $\{f\leq -C^{-1}\}$ , we have $\psi_0=0$ $\mu$-a.e. on $\{f\leq - C^{-1}\}$.  Since $-C^{-1}<f<0$ but $f\cdot \psi_0=1$ on $\{ f >-C^{-1}\}$, we have $\psi_0 < - C$ $\mu$-a.e. on $\{ f >-C^{-1}\}$. Therefore, $f\psi_0 = I_{\{f>-C^{-1}\}} < - C^{-1}\psi_0 $ $\mu$-a.e. and thus $ C \Vert f\psi_0\Vert_{L^2(\mathcal{S},\mu)}< \Vert \psi\Vert_{L^2(\mathcal{S},\mu)}$, which contradicts assumption $(b)$. Finally,
$$ \Vert\nabla h\Vert^2_{L^2(P)}=P(\widehat{\mathcal{L}}_P h\cdot h)=\mu(f\mathcal{U}h^2)\geq C^{-1}\mu(\mathcal{U}h^2)=C^{-1}\Vert h\Vert^2_{L^2(P)}. $$
\end{proof}

\subsection{Proof of Theorem \ref{SCC}}\label{PrfSCC}

\begin{proof}
The case (C1) follows directly from Prop \ref{RglPDE}. For the case (C2), it suffices to show that $P(\mathcal{L}_P f)=0$ for $f\in \mathcal{D}^k(M;P)$. This is from the special Stokes's theorem \cite{gaffney1954special} for differential forms on complete manifold without compact support. For (C3), we establish a lemma, that we need, before proceeding with the proof.

\begin{lemma}\label{boundary} For each smooth function $h_\partial$ on $\partial M$, there exists $h\in C^\infty(M)$ such that $h|_{\partial M}=h_\partial$ and $g(\vec{n},\nabla h)=0$ on $\partial M$, that is, $h\in\mathcal{D}^\infty(M;P)$.
\end{lemma}
\begin{proof} For any boundary coordinate chart $(U,\psi)$ with $\psi(U)$ convex in $\mathbb{R}^n_+$, let $f:=h_{\partial}\circ\psi^{-1}$ and $X=\psi_*\vec{n}$ on $\psi(U)\cap\partial\mathbb{R}^n_+$. Decompose $X$ into $X_\perp + X'$, where $X'\in\partial\mathbb{R}^n_+$ and $X_\perp\perp \partial\mathbb{R}^n_+$. Note that $X_\perp\neq 0$, since $\vec{n}\perp \partial M$ on $M$. Now let
$$ \Bar{f}(x_1,\cdots,x_{n-1},x_n)=f(x_1,\cdots,x_{n-1},0)- \frac{x_n}{\Vert X_\perp\Vert}\cdot \frac{\partial f}{\partial X'}(x_1,\cdots,x_{n-1},0), $$
which satisfies $\frac{\partial \Bar{f}}{\partial X}=0$ on $\partial\mathbb{R}^n$. So $h=\Bar{f}\circ \psi $ is a local extension that is needed.

Take finitely many local charts $\{(U_k,\psi_k)\}_{k=1}^m$ such that $\partial M\subset\bigcup_{k=1}^m U_k$ and each $\psi_k(U_k)$ is convex. Next take a smooth partition of unity $\{\phi_k\}$ subordinate to $(U_k,\psi_k)$ with $\sum_{k=1}^n \phi_k \equiv 1$ on a open neighborhood of $\partial M$. 
For each chart $(U_k,\psi_k)$, there exists a local extension $h_k$, thus $\sum h_k\phi_k$ will be an admissible global extension.
\end{proof}
\medskip

Now we are ready to prove the theorem. Consider the case when $$ \mathcal{H}_P=\mathcal{D}^\infty(M)=\left\{f\in C^\infty(M): g(\vec{n},\nabla f)=0 \right\},$$
as defined as stated in the theorem. By the \nameref{DivThm} and proposition \ref{WPIcase}, $\mathcal{L}_P$ is clearly extendable on $\mathcal{D}^\infty(M)$ and the WPI holds on $\mathcal{D}^\infty(M)$. By proposition \ref{RglPDE}, for each $h\in C^\infty(M)$, there exists a $f\in C^{k+2}(M)\cap\widehat{\mathcal{D}}^\infty(M)$ such that $f$ solves $\mathcal{L}_P f=h$. By the property \ref{FriPro} of Friedrichs extension, $ P(\mathcal{L}_P f, h)=P(\nabla f,\nabla h) $ for $h\in \mathcal{D}^\infty(M;P)$. The \nameref{DivThm} shows that for all $h\in \mathcal{D}^\infty(M;P)$
$$\int_{\partial M} g(\vec{n}, \nabla f)\cdot h e^{-\phi} d\partial v  =  P(\mathcal{L}_P f,h)+P(\nabla f,\nabla h)=0. $$  This implies that $g(\vec{n},\nabla f)\cdot h=0$ for all $h\in\mathcal{D}^\infty(M)$ on $\partial M$. By lemma \ref{boundary}, all smooth functions on $\partial M$ have an extension in $\mathcal{D}^\infty(M)$, which actually implies that $g(\vec{n},\nabla f)=0$ on $\partial M$. Therefore, $f\in \mathcal{D}^k(M)$.
\end{proof}
\end{appendix}

\bibliography{references.bib}
\bibliographystyle{plain}

\end{document}